\documentclass{birkjour}
 \newtheorem{thm}{Theorem}[section]
 
 \newtheorem{lem}[thm]{Lemma}
 
 \theoremstyle{definition}
 
 \theoremstyle{remark}
 \newtheorem{rem}[thm]{Remark}
 \newtheorem{ex}[thm]{Example}
 \numberwithin{equation}{section}

\begin{document}

%-------------------------------------------------------------------------
% editorial commands: to be inserted by the editorial office
%
%\firstpage{1} \volume{228} \Copyrightyear{2004} \DOI{003-0001}
%\seriesextraline{This is the Concrete Title of this Book\br H.E. R and S.T.C. W, Eds.}
%\firstpage{1}

%----------Author 1
\author[Youssef Aserrar, Abdellatif Chahbi and Elhoucien Elqorachi]{Youssef Aserrar, Abdellatif Chahbi and Elhoucien Elqorachi}

\address{% 
	Ibn Zohr University, Faculty of sciences, 
Department of mathematics,\\
Agadir,
Morocco}

\email{youssefaserrar05@gmail.com, abdellatifchahbi@gmail.com,\\ elqorachi@hotmail.com }



\date{January 1, 2020}
\title[A variant of Wilson's functional equation on semigroups]{A variant of Wilson's functional equation on  semigroups}
 
\begin{abstract}
We determine the complex-valued solutions of the following functional equation 
\[f(xy)+\mu (y)f(\sigma (y)x) = 2f(x)g(y),\quad x,y\in S,\]  
where $S$ is a semigroup  and 
 $\sigma$ an automorphism, $\mu :S\rightarrow \mathbb{C}$ is a multiplicative function  such that
$\mu (x\sigma (x))=1$ for all $x\in S$. 
\end{abstract}

%%% ----------------------------------------------------------------------
\maketitle

%%% ----------------------------------------------------------------------
%\tableofcontents
\section{Introduction}
Stetkaer \cite{ST2} solved the variant of d'Alembert's functional equation
\[f(xy)+f(\sigma (y)x)=2f(x)f(y),\quad x,y\in S,\]
where $S$ is a semigroup, and $\sigma : S\rightarrow S$ is an involutive automorphism, that is $\sigma (xy)=\sigma (x)\sigma (y)$ and $\sigma (\sigma (x))=x$, for all $x,y\in S$. The solutions are abelian and are of the form $f=\dfrac{\chi +\chi \circ \sigma}{2}$, where $\chi : S\rightarrow \mathbb{C}$ is a multiplicative function of $S$.\newline
In \cite{EA}  Elqorachi and Redouani determined the complex-valued solutions of the variant of Wilson's functional equation
\begin{equation}
f(x y)+\mu(y) f(\sigma(y) x)=2f(x)g(y), \quad x, y \in G,
\label{A2}
\end{equation}
where $G$ is a group and $\mu : S\rightarrow \mathbb{C}$ is a multiplicative function such that $\mu (x\sigma (x))=1$ for all $x\in G$.  Fadli et al. \cite{F1} obtained the solutions of \eqref{A2} with $\mu =1$ on groups. In \cite{SA1} Sabour determined the solutions of \eqref{A2} on groups with $\mu=1$ and $\sigma$ an automorphism not necessarily involutive.\newline
In a recent paper \cite{Ajb} Ajebbar and Elqorachi solved \eqref{A2} on semigroups generated by their squares. \\
Recently, Ebanks \cite{EB2} obtain the solution of the partially Pexiderized d'Alembert-
type equation $f (x\sigma(y)) + h(\tau (y)x) = 2f (x)k(y)$ for $f, g, h, k : M \rightarrow \mathbb{C}$,
where $\sigma, \tau$ are involutive automorphisms on monoids  that are neither regular nor generated by their squares. There are some results about solutions of equation \eqref{A2} with $\mu=1$ on abelian groups in the literature. See Aczél, Dhombres \cite{Acz}, Stetkaer \cite{ST} for further contextual and historical discussion.\newline
Our attention was drawn to \eqref{A2} because in its solutions on groups and semigroups the sine addition law 
\begin{equation} 
f(xy)= f(x)g(y)+f(y)g(x),\quad x,y\in S,
\label{moneq}
\end{equation}
 plays an important role, and in the recent papers \cite{EB1,EB2} by Ebanks the solutions of \eqref{moneq} are described in a general semigroup.\par
 The contributions of the present paper to the knowledge about solutions of
 \eqref{A2} are the following :\\
 1) The setting has $S$ to be a semigroup, not necessarily generated by its squares.\\2) The automorphism  
  $\sigma : S\rightarrow S$ is not necessary involutive. \\3) We relate the solutions of Wilson's functional equation  \eqref{A2} to the sine addition law \eqref{moneq}, and we find explicit formulas for the solutions, expressing them in terms of multiplicative, additive functions, and other arbitrary functions.
\section{Notations and Terminology}
We impose as blanket assumptions that ($S, .$) is a semigroup. We say that the semigroup $S$ is a topological semigroup, if $S$ is equipped
with a topology such that the product map $(x,y)\mapsto xy$ from $S \times S$ to $S$ is
continuous, when $S \times S$ is given the product topology.\par
We say that a function $f$ on $S$
 is additive if $f(xy) = f(x) + f(y)$ for all $x, y\in S$; 
$f$ is multiplicative if $f(xy) = f(x)f(y)$ for all $x, y \in S$, 
$f$ is central if $f(xy) = f(yx)$ for all $x, y\in S$, and $f$ is abelian if $f$ is central  and $f (xyz) =
f (xzy)$ for all $x, y, z \in S.$

If $S$ is a semigroup, $\sigma: S \rightarrow S$ an automorphism and $\mu: S \rightarrow \mathbb{C}$ a multiplicative function such that $\mu(x \sigma(x))=1$ for all $x \in S$.
For any subset $T\subseteq S$ we define $T^2:=\lbrace xy\quad \vert \quad x, y\in T\rbrace$.
If $\chi: S \rightarrow \mathbb{C}$ is a multiplicative function and $\chi \neq 0$, then $I_{\chi}:=\{x \in S \mid \chi(x)=0\}$ is either empty or a proper subset of $S$.  $I_{\chi}$ is a two sided ideal in $S$ if not empty and $S \backslash I_{\chi}$ is a subsemigroup of $S$. Notice that $I_{\chi}$ is also a subsemigroup of $S$. We define as in \cite{EB2} the set $P_\chi : =\{p\in I_{\chi}\backslash I_{\chi}^2\quad \vert up, pv, upv\in I_{\chi}\backslash I_{\chi}^2\quad \text{for all}\quad u,v\in S\backslash I_{\chi}\}$.\newline
 For any function $f: S \rightarrow \mathbb{C}$ we note $f^{*}(x):=\mu(x) f(\sigma(x)),\quad \text{for all}\quad x \in S$, and the functions $f^{e}:=\frac{f+f^{*}}{2}$, $f^{\circ}:=\frac{f-f^{*}}{2}$.\par
 For a topological semigroup $S$ let $C(S)$ denote the algebra of continuous functions mapping $S$ into $\mathbb{C}$.
\section{Main result}

In the following lemma we give some key properties of  solutions of Equation \eqref{A2}.
\begin{lem}
\label{Lemma}
Let $f, g : S\rightarrow \mathbb{C}$ be a solution of Equation \eqref{A2}. The follwoing statements hold:
\begin{enumerate}
\item[(1)] $f_a(xy)=f_a(x)g(y)+f_a(y)g(x)$, for all $a,x,y \in S$, where $f_a(x)=f(ax)-f(a)g(x)$ for all $x\in S$, and hence $f_a$ and $g$ are abelian, in particular central.
\item[(2)] $f^{\circ}(xy)=f(x)g(y)-f^*(y)g(x)$, for all $x,y \in S$.
\item[(3)] $f(xy)=2f(x)g(y)+2f(y)g(x)-4f^e(y)g(x)+f^*(xy)$, for all $x,y \in S$.
\item[(4)] If $f$ is central and $g\neq 0$ then $f^e$ and $g$ are linearly dependent.
\item[(5)] If $f^e$ and $g$ are linearly independent and $g\neq 0$, then there exists two functions $h_1, h_2 : S\rightarrow \mathbb{C}$ such that 
\[g(xy)=f(x)h_1(y)+g(x)h_2(y),\quad\text{for all}\quad x,y\in S.\]
\item[(6)] If $g$ is a non-zero multiplicative function then there exists a function $h : S\rightarrow \mathbb{C}$ such that 
\[f(xy)=f(x)g(y)+g(x)h(y),\quad\text{for all}\quad x,y\in S.\]
\end{enumerate}

\end{lem}
\begin{proof}
(1) Let $a,x,y \in S$ be arbitrary. We use similar computations to those of \cite{ST2}.\newline
We apply equation \eqref{A2} to the pair $(ax,y)$, we obtain 
\[f(axy)+\mu (y)f(\sigma(y)ax)=2f(ax)g(y).\]
Now if we apply \eqref{A2} to the pair $(\sigma (y)a,x)$ and multiply the identity obtained by $-\mu (y)$ we get
\[-\mu (y)f(\sigma (y)ax)-\mu (xy)f(\sigma (x)\sigma (y)a)=-2\mu (y) f(\sigma (y)a)g(x).\]
For the pair $(a,xy)$ Equation \eqref{A2}  becomes
\[f(axy)+\mu (xy)f(\sigma(x)\sigma (y)a)=2f(a)g(xy).\]
By adding these three identities we obtain 
\[f(axy)=f(a)g(xy)+f(ax)g(y)+g(x)[f(ay)-2f(a)g(y)].\]
Since $a,x,y$ are arbitrary we deduce that for all $a\in S$ the pair $(f_a,g)$ satisfies the sine addition law 
\begin{equation}
f_a(xy)=f_a(x)g(y)+f_a(y)g(x), \quad x,y \in S.
\label{B1}
\end{equation}
Equation \eqref{B1} and \cite[Theorem 3.1]{EB2} implies that $f_a$ and $g$ are abelian, in particular central. This is case (1).\\
(2) By applying \eqref{A2} to the pair $(\sigma (y),x)$ and multiplying the identity obtained by $-\mu (y)$ we get 
\begin{equation}
-\mu (y)f(\sigma (y)x)-f^*(xy)=-2f^*(y)g(x).
\label{B2}
\end{equation}
By adding \eqref{B2} to \eqref{A2} we obtain 
\begin{equation}
f(xy)-f^*(xy)=2f(x)g(y)-2f^*(y)g(x).
\label{B3}
\end{equation}
This implies that 
\[f^{\circ}(xy)=f(x)g(y)-f^*(y)g(x).\]
This is the result (2) of Lemma \ref{Lemma}.\\
(3) The identity \eqref{B3} implies that
\[f(xy)=2f(x)g(y)-2f^*(y)g(x)+f^*(xy).\]
Since $f^*=2f^e-f$ we get 
\[f(xy)=2f(x)g(y)-2(2f^e(y)-f(y))g(x)+f^*(xy).\]
So
\begin{equation}
f(xy)=2f(x)g(y)+2f(y)g(x)-4f^e(y)g(x)+f^*(xy).
\label{B4}
\end{equation} 
This occurs in part (3).\\
(4) If $f$ is central then $f^*$ is also central, so taking this into account in the identity \eqref{B4} we deduce that 
\[-4f^e(y)g(x)=-4f^e(x)g(y),\quad\text{for all}\quad x,y\in S.\]
Since $g\neq 0$ then there exists $x_0 \in S$ such that $g(x_0)\neq 0$. By replacing $x$ by $x_0$ and putting $c=\dfrac{f^e(x_0)}{g(x_0)}$ in the identity above we get 
\[f^e(y)=cg(y),\quad\text{for all}\quad y\in S.\]
This is part (4).\\
(5) Suppose that $f^e$ and $g$ are linearly independent and $g\neq 0$. Using the associativity of the semigroup operation, we can compute $f^{\circ}(xyz)$ first as $f^{\circ}(x(yz))$ and then as $f^{\circ}((xy)z)$ using the identity in (2) and compare the results. We obtain
\begin{equation}
f(xy)g(z)-f^*(z)g(xy)=f(x)g(yz)-f^*(yz)g(x).
\label{B5}
\end{equation}
Since $f=f^e+f^{\circ}$ and $f^*=f^e-f^{\circ}$ then by using the identity in (2) we get
\[f(xy)=f^e(xy)+f(x)g(y)-f^*(y)g(x)\] and 
\[f^*(yz)=f^e(yz)-f(y)g(z)+f^*(z)g(y).\]
Substituting the last two identities in \eqref{B5} we get after some rearrangement
\begin{align*}
g(z)\left[ g(x)f^e(y)-f^e(xy)\right] +f^*(z)\left[ g(xy)-g(x)g(y)\right]= \\f(x)\left[ g(y)g(z)-g(yz)\right] +g(x)\left[ f^e(yz)-f^e(y)g(z)\right].
\end{align*}
If $f^*$ and $g$ are linearly dependent then $f^*$ is central, since $\sigma$ is an automorphism we deduce that $f$ is also central, according to (4) this implies that $f^e$ and $g$ are linearly dependent, this is a contradiction. So $f^*$ and $g$ are linearly independent.\newline
By fixing $z=z_1$ and $z=z_2$ such that $g(z_1)f^*(z_2)-g(z_2)f^*(z_1)\neq 0$ in the identity above we obtain two equations from which we get 
\begin{equation}
g(xy)=f(x)h_1(y)+g(x)h_2(y),
\end{equation}
for some functions $h_1$ and $h_2$. This occurs in (5).\\
(6) If $g$ is a non-zero multiplicative function then \eqref{B5} becomes
\[f(xy)g(z)-f^*(z)g(x)g(y)=f(x)g(y)g(z)-f^*(yz)g(x).\]
This implies that 
\begin{equation}
g(z)\left( f(xy)-f(x)g(y)\right)=g(x)\left(f^*(z)g(y)-f^*(yz) \right).  
\label{B6}
\end{equation}
By fixing $z=z_0$ such that $g(z_0)\neq 0$ in \eqref{B6} we deduce that 
\begin{equation}
f(xy)=f(x)g(y)+g(x)h(y),
\end{equation}
for some function $h$. This is part (6). This completes the proof of Lemma \ref{Lemma}.
\end{proof}
In the following lemma we give some properties of the subsets $P_{\chi}$ and $I_{\chi}\backslash P_{\chi}$ when $\chi$ is $\sigma$-invariant (See \cite[Lemma 4.1]{EB1}).
\begin{lem}
Let $\chi :S\rightarrow \mathbb{C}$ be a non-zero multiplicative function such that $\chi\circ\sigma=\chi$, and $\sigma :S\rightarrow S$ an automorphism. Then
\begin{enumerate}
\item[(1)] $\sigma (P_{\chi})\subset P_{\chi}$.
\item[(2)] $\sigma (I_{\chi}\backslash P_{\chi})\subset I_{\chi}\backslash P_{\chi}$.
\end{enumerate} 
\label{le}
\end{lem}
\begin{proof}
(1) Let $x\in P_{\chi}$, suppose that $\sigma(x)\notin P_{\chi}$, then there exists $y\in S\backslash I_{\chi}$ such that $\sigma(x)y\in I_{\chi}^2$. Since $\sigma$ is an automorphism and $\chi=\chi\circ\sigma$, then there exists $z\in S\backslash I_{\chi}$ such that $y=\sigma(z)$, so $\sigma(x)y=\sigma(x)\sigma(z)=\sigma(a)\sigma(b)$, for some $a,b\in I_{\chi}$. Then $xz=ab\in I_{\chi}^2$, which implies that $x\notin P_{\chi}$ but this is a contradiction, and then $\sigma(x)\in P_{\chi}$.\\
(2) For all $x\in I_{\chi}\backslash P_{\chi}$, $\sigma (x)\in I_{\chi}$ since $\chi=\chi\circ\sigma$. Suppose that $\sigma(x)\in P_{\chi}$, then for all $y\in S\backslash I_{\chi}$ we have $\sigma(x)y\in I_{\chi}\backslash I_{\chi}^2$. Since $\sigma$ is an automorphism then there exists $z\in S\backslash I_{\chi}$ such that $y=\sigma(z)$, so $\sigma(x)y=\sigma(xz)\in I_{\chi}\backslash I_{\chi}^2$, then since $\chi=\chi\circ\sigma$ we get that $xz\in I_{\chi}\backslash I_{\chi}^2$, since $y$ is arbitrary then $z$ is also arbitrary. This implies that $x\in P_{\chi}$, this is a contradiction, and then $\sigma(x)\in I_{\chi}\backslash P_{\chi}$. This completes the proof of Lemma \ref{le}.
\end{proof}
Now we are ready  to solve the functional equation \eqref{A2}.
\begin{thm}
\label{Th}
The solutions $f, g :S \rightarrow \mathbb{C}$ of the functional equation \eqref{A2} with $g\neq 0$ are the following pairs:
\begin{enumerate}
\item[(1)] $f=0$ and $g\neq 0$  arbitrary.
\item[(2)] $f=\alpha \chi +\beta \chi ^*$ and $g=\dfrac{\chi +\chi ^*}{2}$ where  $\chi: S \rightarrow \mathbb{C}$ is a non-zero multiplicative function  and $\alpha ,\beta \in \mathbb{C}$ are constants such that $\left(\alpha ,\beta \right)\neq (0,0) $ and $\chi^* \neq \chi$, in addition if  $\beta \neq 0$ then $\chi\circ \sigma^2=\chi$.
\item[(3)] $$f=\left\{ \begin{matrix}
   \chi (c+A) & on & S\backslash {{I}_{\chi }}  \\
   0 & on & {{I}_{\chi }}\backslash {{P}_{\chi }}  \\
   \rho  & on & {{P}_{\chi }}  \\
\end{matrix} \right.\quad \text{and}\quad g=\chi,$$ 
\end{enumerate}
where $c \in \mathbb{C}$ is a constant, $\chi: S \rightarrow \mathbb{C}$ is a non-zero multiplicative function and $A: S \backslash I_{\chi} \rightarrow \mathbb{C}$ is an additive function such that $\chi^{*}=\chi$ and $A \circ \sigma=-A$, $\rho: P_{\chi} \rightarrow \mathbb{C}$ is the restriction of $f$ to $P_\chi$ such that at least one of  $c, A$ and $\rho$ is not zero and $\rho ^*=-\rho$. In addition we have the following conditions:\\
(I) : If $x\in\left\lbrace up,pv,upv \right\rbrace $ for $p\in P_{\chi}$ and $u,v\in S\backslash I_{\chi}$, then $x\in P_{\chi}$ and we have respectively $\rho(x)=\rho(p)\chi (u)$, $\rho(x)=\rho(p)\chi (v)$, or $\rho(x)=\rho(p)\chi (uv)$.\\
(II) : $f(xy)=f(yx)=0$ for all $x\in S\backslash I_{\chi}$ and $y\in I_{\chi}\backslash P_{\chi}$.\\
Note that, off the exceptional case (1) $f$ and $g$ are abelian.\par
Furthermore, off the exceptional case (1) if $S$ is a topological semigroup and $f\in C(S)$, then $g,\chi,\chi^*\in C(S)$,  $A\in C(S\backslash I_{\chi})$ and $\rho \in C(P_{\chi})$.
\end{thm}
\begin{proof}
We check by elementary computations that if $f$, $g$ are of the forms (1)--(3) then $(f,g)$ is a solution of \eqref{A2}, so left is that any solution $(f,g)$ of \eqref{A2} fits into (1)--(3).\newline
Let $f, g :S \rightarrow \mathbb{C}$ be a solution of \eqref{A2}. If $f=0$ then $g$ is arbitrary. From now on we assume that $f\neq 0$. According to Lemma \ref{Lemma} (2) we have 
\begin{equation}
f^{\circ}(xy)=f(x)g(y)-f^*(y)g(x),\quad\text{for all}\quad x,y\in S.
\label{par}
\end{equation}
Since $f=f^e+f^{\circ}$ and $f^*=f^e-f^{\circ}$, equation \eqref{par} can be written as follows 
\begin{equation}
f^{\circ}(xy)=f^{\circ}(x)g(y)+f^{\circ}(y)g(x)+f^e(x)g(y)-f^e(y)g(x).
\label{B7}
\end{equation}
\underline{First case :}  $f^e$ and $g$ are linearly dependent. There exists a constant $c\in \mathbb{C}$ such that $f^e=cg$, then for all $x,y\in S$ we have 
\[f^e(x)g(y)-f^e(y)g(x)=cg(x)g(y)-cg(y)g(x)=0.\]
Now equation\eqref{B7} becomes
\begin{equation}
f^{\circ}(xy)=f^{\circ}(x)g(y)+f^{\circ}(y)g(x),\quad\text{for all}\quad x,y\in S.
\label{B8}
\end{equation}
Equation \eqref{B8} means that the pair $(f^{\circ},g)$ satisfies the sine addition law, then according to \cite[Theorem 3.1]{EB2} and taking into account that $g\neq 0$ we have the following possibilities:\\
(i) $f^{\circ}=c_1\left( \chi _1 -\chi _2\right) $ and $g=\dfrac{\chi _1+\chi _2}{2}$, for some constant $c_1\in \mathbb{C}$ and $\chi_1, \chi _2 : S\rightarrow \mathbb{C}$ are two  multiplicative functions such that $\chi _1 \neq \chi _2$. So since $f^e=cg$, we deduce that $f=\alpha \chi _1+\beta \chi _2$ for some constants $\alpha ,\beta \in \mathbb{C}$. Substituting $f$ and $g$ in the functional equation \eqref{A2} we get after some simplification that for all $x,y\in S$
\[\alpha \chi _1 (x)\left[ \chi_1 ^*(y)-\chi _2 (y)\right]+\beta \chi _2(x)\left[\chi _2^*(y)-\chi_1(y) \right] =0.\]
Since $\chi_1\neq \chi_2$ then according to \cite[Theorem 3.18]{ST} we get that 
$$
\left\{\begin{array}{l}
\alpha \chi _1 (x)\left[ \chi_1 ^*(y)-\chi _2 (y)\right]=0 \\
\beta \chi _2(x)\left[\chi _2^*(y)-\chi_1(y) \right] =0 
\end{array},\right.
$$
for all $x,y\in S$. Since $f\neq 0$ then at least one of $\alpha $ and $\beta $ is not zero.\newline
If $\alpha \neq0$ and $\beta =0$ then we deduce that $\chi _1 \neq 0$, $\chi_2\neq 0$ and $\chi_1^*(y)=\chi_2(y)$ for all $y\in S$. The result occurs in part (2) with $\beta =0$ and $\chi_1=\chi$.\newline
If $\alpha \neq0$ and $\beta \neq 0$ then if $\chi _1 = 0$ then $\chi_2\neq 0$, and then $\chi_1=\chi_2^* \neq 0$, this is a contradiction, so $\chi_1 \neq 0$ and  $\chi_1^*(y)=\chi_2(y)$, then  $\chi_2^*(y)=\chi_1(y)$ for all $y\in S$. Now if we put $\chi =\chi_1$ we have $\chi^*=\chi_2$ and $\chi \circ \sigma^2=\chi$.\newline
If $\alpha =0$ and $\beta \neq 0$ then $\chi _1 \neq 0$, $\chi_2\neq 0$,  $\chi_2^*(y)=\chi_1(y)$ for all $y\in S$.  This occurs in part (2) with $\alpha =0$, $\chi_1=\chi$ and $\chi_2=\chi^*$, so we have $\chi \circ \sigma^2=\chi$.\\
(ii) $$f^{\circ}=\left\{ \begin{matrix}
   \chi A & on & S\backslash {{I}_{\chi }}  \\
   0 & on & {{I}_{\chi }}\backslash {{P}_{\chi }}  \\
   \rho  & on & {{P}_{\chi }}  \\
\end{matrix} \right.\quad \text{and}\quad g=\chi,$$ 
where  $\chi: S \rightarrow \mathbb{C}$ is a non-zero multiplicative function and $A: S \backslash I_{\chi} \rightarrow \mathbb{C}$ an additive function, $\rho: P_{\chi} \rightarrow \mathbb{C}$ is a function satisfying the condition (I), and $f^{\circ}$ satisfies the condition (II). Since $f^e=c\chi=0$ on $I_{\chi}$, then $f=f^{\circ}$ on $I_{\chi}$, so $f$ satisfies the condition (II). Since $f=f^{\circ}+f^e=f^{\circ}+cg$, we obtain
$$f=\left\{ \begin{matrix}
   \chi (c+A) & on & S\backslash {{I}_{\chi }}  \\
   0 & on & {{I}_{\chi }}\backslash {{P}_{\chi }}  \\
   \rho  & on & {{P}_{\chi }}  \\
\end{matrix} \right..$$  
By applying the identity \eqref{par} to the pair $(\sigma(y),x)$ and multiplying the identity obtained by $\mu (y)$, we get 
\begin{equation}
\mu(y)f^{\circ}(\sigma(y)x)=f^*(y)g(x)-f^*(x)g^*(y).
\label{par1}
\end{equation}
By adding \eqref{par1} to \eqref{par}, we get that 
\begin{equation}
f^{\circ}(xy)+\mu(y)f^{\circ}(\sigma(y)x)=f(x)g(y)-f^*(x)g^*(y).
\label{par2}
\end{equation}
Now by subtracting \eqref{par2} from \eqref{A2} we get
$$f^e(xy)+\mu(y)f^e(\sigma(y)x)=f(x)g(y)+f^*(x)g^*(y).$$
Since $f^e=c\chi$, we deduce that 
$$c\chi(x)\chi(y)+c\chi(x)\chi^*(y)=f(x)\chi(y)+f^*(x)\chi^*(y).$$
This implies that 
\begin{equation}
\chi(y)\left[c\chi(x)-f(x) \right]+\chi^*(y)\left[c\chi(x)-f^*(x) \right].
\label{par3}  
\end{equation}
If $\chi \neq \chi^*$, then since $\chi$ and $\chi^*$ are non-zero, we get from \eqref{par3} that 
$$f(x)=c\chi(x)\quad \text{and}\quad f^*(x)=c\chi(x),$$
for all $x\in S$. Since $f\neq 0$, then $c\neq 0$ and $f^*=c\chi^*=c\chi$, and then $\chi=\chi^*$. This is a contradiction. So $\chi=\chi^*$ and the functional equation \eqref{A2} implies that 
\[\chi (xy)\left(c+A(xy) \right)+\mu (y)\chi (\sigma (y)x)\left(c+A(\sigma (y)x) \right)=2\chi (y)\chi (x)\left(c+A(x) \right),   \]
for all $x,y\in S\backslash I_{\chi}$. Since $A$ is additive and $\chi (xy)\neq 0$, the identity above reduces to $A\circ \sigma =-A$. On the other hand, $f\neq 0$ implies that at least one of  $c, A$ and $\rho$ is not zero. For $x\in S\backslash I_{\chi}$ and $y\in P_{\chi}$ we have $xy\in P_{\chi}$ and  by Lemma 3.2 (1) we get $\sigma(y)\in P_{\chi}$, so $\sigma (y)x\in P_{\chi}$, then Equation \eqref{A2} can be written as follows
\[\rho (y)\chi (x)+\rho^*(y)\chi (x)=0,\]
 which implies that $\rho=-\rho^*$ since $\chi$ is non-zero. Now if $y\in I_{\chi}\backslash P_{\chi}$, then by Lemma 3.2 (2) we get $\sigma(y)\in I_{\chi}\backslash P_{\chi}$, and we get by the condition (II) that $f(xy)=f(\sigma (y)x)=0$, so
 \[f(xy)+\mu (y)f(\sigma (y)x)=0=2f(x)\chi(y),\]
since $\chi(y)=0$. This is part (3) of Theorem \ref{Th}.\newline
\underline{Second case :}  $f^e$ and $g$ are linearly independent. According to Lemma \ref{Lemma} (5) there exists two functions $h_1, h_2 :S \rightarrow \mathbb{C}$ such that 
\begin{equation}
g(xy)=f(x)h_1(y)+g(x)h_2(y) \quad \text{for all}\quad x,y \in S.
\label{B10}
\end{equation} 
According to Lemma \ref{Lemma} (2),  $(f_a,g)$ satisfies the sine addition law (\ref{moneq}) then we have the following possibilities :\newline
\underline{Subcase A :}  $f_a=0$ for all $a\in S$. That is $f(xy)=f(x)g(y)$ for all $x,y\in S$, according to the proof of \cite[Theorem 4.2 (case 1)]{Ajb}, this case leads to $f=\lambda \chi$ and $g=\chi$ where $\chi : S\rightarrow \mathbb{C}$ is non-zero multiplicative function and $\lambda \in \mathbb{C}$ a constant. We deduce that $f$ is central, then since $g\neq 0$ according to Lemma \ref{Lemma} (4) $f^e$ and $g$ are linearly dependent. This case does not occur.\newline
\underline{Subcase B :}  $f_a \neq 0$ for some $a\in S$. There exists two  multiplicative functions $\chi_ 1 ,\chi _2 :S\rightarrow \mathbb{C}$ such that $g=\dfrac{\chi _1+\chi _2}{2}$.\newline
\underline{Subcase B.1 :}  $\chi _1 \neq \chi _2$. \newline
\underline{Subcase B.1.1 :}  $h_1 =0$. The identity \eqref{B10} becomes 
\begin{equation}
g(xy)=g(x)h_2(y),\quad\text{for all}\quad x,y\in S.
\label{B11}
\end{equation}
Since $g$ is central and $g\neq 0$, we deduce from \eqref{B11} that $h_2=bg$ for some constant $b\in \mathbb{C}$, so \eqref{B11} can be written as follows 
\begin{equation}
g(xy)=bg(x)g(y),\quad\text{for all}\quad x,y\in S.
\label{B12}
\end{equation}
Since $g=\dfrac{\chi _1+\chi _2}{2}$ then we deduce from the identity \eqref{B12} that for all $x,y\in S$
\begin{equation}
(2-b)\left( \chi_1(xy)+\chi_2(xy)\right)=b\left( \chi_1(x)\chi_2(y)+\chi_1(y)\chi_2(x)\right). 
\label{B13} 
\end{equation}
Since $g\neq 0$ then there exists $y_0\in S$ such that $\chi_2(y_0)\neq 0$, so if we put $y=y_0$ in the identity \eqref{B13} we obtain that for all $x\in S$
\begin{equation}
\left[(2-b)\chi_1(y_0)-b\chi_2(y_0) \right] \chi_1(x)+\left[(2-b)\chi_2(y_0)-b\chi_1(y_0) \right] \chi_2(x)=0.
\label{B14}
\end{equation}
Since $\chi_1\neq \chi_2$ then by using \cite[Theorem 3.18]{ST} we deduce from \eqref{B14} that 
\begin{equation}
\left[ (2-b)\chi_1(y_0)-b\chi_2(y_0)\right] \chi_1=0, 
\label{K1}
\end{equation}
\begin{equation}
\left[ (2-b)\chi_2(y_0)-b\chi_1(y_0)\right] \chi_2=0.
\label{K2}
\end{equation}

If $b=0$ then we get from \eqref{K2} that $\chi_2=0$, this is a contradiction. So $b\neq 0$. If $b=2$ then we get from \eqref{K1} that $\chi_1=0$, that is $g=\dfrac{\chi_2}{2}$. Then Equation \eqref{B5} implies that  for all $x,y,z\in S$ 
\begin{equation}
\chi_2 (z)\left[f(xy)-\chi_2(y)f(x) \right] =\chi_2(x)\left[\chi_2(y)f^*(z)-f^*(yz) \right].
\label{rej1} 
\end{equation}
Since $\chi_2\neq 0$, there exists $z_0\in S$ such that $\chi_2(z_0)\neq 0$. By putting $z=z_0$ in \eqref{rej1}, we get that 
\begin{equation}
f(xy)=\beta \chi_2(xy)+\chi_2(x)k(y)+\chi_2(y)f(x),\  \text{for all}\quad x,y\in S,
\label{rej2}
\end{equation}
where $k(y)=\dfrac{-f^*(yz_0)}{\chi_2(z_0)}$ and $\beta =\dfrac{f^*(z_0)}{\chi_2(z_0)}\in \mathbb{C}$ is a constant. By applying the identity \eqref{rej2} to the pair $(\sigma (y),x)$ and multiplying the identity obtained by $\mu (y)$, we get 
\begin{equation}
\mu (y)f(\sigma (y)x)=\beta \chi_2^*(y)\chi_2(x)+\chi_2^*(y)k(x)+\chi_2(x)f^*(y),\quad x,y\in S.
\label{rej3}
\end{equation}
By adding \eqref{rej2} to \eqref{rej3} and taking into account that the pair $(f,g)$ satisfies \eqref{A2}, we get after some rearrangement 
\begin{equation}
\chi_2(x)\left[\beta (\chi_2(y)+\chi_2^*(y))+f^*(y)+k(y) \right] =-\chi_2^*(y)k(x).
\label{rej4}
\end{equation}
Since $\chi_2\neq 0$, then we deduce from \eqref{rej4} that $f^*+k=a_1\chi_2+a_2\chi_2^*$ for some constants $a_1,a_2\in \mathbb{C}$. Taking this into account in \eqref{rej4}, we get that 
\begin{equation}
\chi_2(x)\left[b_1\chi_2(y)+b_2\chi_2^*(y) \right] =-\chi_2^*(y)\left[a_1\chi_2(x)+a_2\chi_2^*(x)-f^*(x) \right],
\label{rej5}
\end{equation}
for some contsants $b_1,b_2\in \mathbb{C}$. Since $\chi_2\neq 0$, then we deduce from \eqref{rej5} that $f^*=c_1\chi_2+c_2\chi_2^*$, where $c_1,c_2\in \mathbb{C}$ are contsants. This implies that $f^*$ is central, and then $f$ is central since $\sigma$ is an automorphism, so according to Lemma 3.1 (4) $f^e$ and $g$ are linearly dependent. This case does not occur. So $b\neq 2$ and $\chi_1 \neq 0$, then we get from equation \eqref{K2} since $\chi_2\neq 0$ that $\chi_1(y_0)=\dfrac{2-b}{b}\chi_2(y_0)\neq 0$. Taking this into account in \eqref{K1}, we deduce that 
\[\left[\dfrac{(2-b)^2}{b} -b\right] \chi_2(y_0)\chi_1=0.\]
Since $\chi_2(y_0)\neq 0$ and $\chi_1\neq 0$, this implies that $(2-b)^2-b^2=0$ then $b=1$, and then we deduce from \eqref{B12} that $g$ is a multiplicative function, this means that $\chi_1 =\chi_2$ but this is a contradiction. This case does not occur.\newline
\underline{Subcase B.1.2 :} $h_1 \neq0$. There exists $y_0 \in S$ such that $h_1(y_0)\neq 0$, then we deduce from the identity \eqref{B10} that $f(x)=c_1g(xy_0)+c_2g(x)$ for all $x\in S$ and some constants $c_1, c_2 \in \mathbb{C}$. Since $g=\dfrac{\chi_1+\chi_2}{2}$ where $\chi_1, \chi_2$ are multiplicative functions, we deduce that $f=a_1\chi_1+a_2\chi_2$ for some constants  $a_1, a_2 \in \mathbb{C}$. This implies that $f$ is central then according to Lemma \ref{Lemma} (4) $f^e$ and $g$ are linearly dependent. This case does not occur.\newline
\underline{Subcase B.2 :}  $\chi _1 =\chi _2$. In this case $g$ is a multiplicative function, then according to Lemma \ref{Lemma} (5) there exists a function $h$ such that 
\begin{equation}
f(xy)=f(x)g(y)+g(x)h(y),\quad\text{for all}\quad x,y\in S.
\label{B15}
\end{equation}
\underline{Subcase B.2.1 : } $h=0$. From \eqref{B15} we deduce that $f(xy)=f(x)g(y)$ for all $x,y\in S$. This implies that $f_a=0$, and this is a contradiction.\newline
\underline{Subcase B.2.2 : } $h\neq 0$. If we put $x=a$ in \eqref{B15} we get $f_a(y)=g(a)h(y)$ for all $y\in S$, since $f_a \neq 0$ then $g(a)\neq 0$, then $h=\dfrac{1}{g(a)}f_a$, this implies that $h$ is central since $f_a$ is central. On the other hand if we apply the identity \eqref{B15} to the pair $(\sigma (y),x)$ and multiply the identity obtained by $\mu (y)$ we get 
\begin{equation}
\mu (y)f(\sigma (y)x)=f^*(y)g(x)+g^*(y)h(x),\quad\text{for all}\quad x,y\in S.
\label{B16}
\end{equation}
By adding \eqref{B15} and \eqref{B16} and taking into account that $(f,g)$ satisfies the functional equation \eqref{A2} we get 
\begin{equation}
f(x)g(y)=g(x)\left[h(y)+f^*(y) \right]+g^*(y)h(x),\quad\text{for all}\quad x,y\in S. 
\label{B17}
\end{equation}
Since $g\neq 0$ there exists $y_0 \in S$ such that $g(y_0)\neq 0$, then we deduce from \eqref{B17} that 
\begin{equation}
f(x)=b_1g(x)+b_2h(x),\quad\text{for all}\quad x\in S,
\label{B18}
\end{equation}
for some constants $b_1,b_2 \in \mathbb{C}$. Since $g$ and $h$ are central, we deduce from \eqref{B18} that $f$ is also central. Then according to Lemma \ref{Lemma} (4) the functions $f^e$ and $g$ are linearly dependent. This case does not occur.\par
For the topological statements suppose that $f$ is continuous and $f\neq 0$. The continuity of $g$ follows easily from the continuity of $f$ and the functional equation \eqref{A2}. Let $y_0\in S$ such that $f(y_0)\neq 0$, we get from \eqref{A2} that
\[g(x)=\dfrac{f(xy_0)+\mu (y_0)f(\sigma(y_0)x)}{2f(y_0)}\quad\text{for}\ x\in S.\]
The functions $x\mapsto f(xy_0)$ and $x\mapsto f(\sigma(y_0)x)$ are continuous, since $S$ is a topological semigroup so that the right translation $x\mapsto xy_0$ and the left translation $x\mapsto \sigma(y_0)x$ are continuous. Then $g$ is continuous.\\
In case (2) we get the continuity of $\chi$ and $\chi^*$ from the continuity of $g$ by the help of \cite[Theorem 3.18]{ST}. For the case (3) the function $\rho$ is continuous by restriction since $f$ is continuous, and 
\[\chi A=f-c\chi\quad\text{on}\quad S\backslash I_{\chi}.\]
$g=\chi$ is continuous. So $A$ is continuous since $\chi\neq 0$. This completes the proof of Theorem \ref{Th}.
\end{proof} 
\begin{rem}
 For a semigroup $S$ such that $S^2\neq S$, there exists a non-zero function $f$ such  that\newline
 \begin{equation}
 f(xy)+\mu (y)f(\sigma (y)x)=0,\quad\text{for all}\quad x,y\in S.
 \label{W0}
 \end{equation}
 Let $S=\left\lbrace 0,1 \right\rbrace $ and define the semigroup operation as $xy=0$, for all $x,y\in S$, and $\sigma (x)=x$ for all $x\in S$, we let $f$ be  the function
 \[f(x)=\left\{\begin{array}{l}
0\quad\text{if}\  x=0\\
1\quad\text{if}\  x=1
\end{array},\right.\] 
 it is easy to verify that $f$ satisfies \eqref{W0}. 
 \end{rem}
\section{ Examples}
In this section we give some examples of the non-zero continuous  solutions of the functional equation \eqref{A2} with $\mu =1$.
\begin{ex}
Let $G$ be the $(ax+b)$--group defined by \newline
\[G:=\left\lbrace \left(\begin{matrix}
   a & b  \\
   0 & 1  \\
\end{matrix} 
 \right)\mid a>0,\quad b\in \mathbb{R}  \right\rbrace .\]
 We consider the following automorphism on $G$ 
 \[\sigma \left(\begin{matrix}
   a & b  \\
   0 & 1  \\
\end{matrix} 
 \right)=\left(\begin{matrix}
   a & 2b  \\
   0 & 1  \\
\end{matrix} 
 \right),\]
 $\sigma$ is not involutive. According to \cite[Example 2.10, Example 3.13]{ST}, the continuous additive and the non-zero multiplicative functions on $G$ have respectively the forms
 \[A_c : \left(\begin{matrix}
   a & b  \\
   0 & 1  \\
\end{matrix} 
 \right)\mapsto c\log(a) ,\]
 and 
 \[\chi_{\lambda} : \left(\begin{matrix}
   a & b  \\
   0 & 1  \\
\end{matrix} 
 \right)\mapsto a^{\lambda},\]
 where $c, \lambda \in \mathbb{C}$. We can see that $\chi_{\lambda}\circ \sigma =\chi_{\lambda}$ and $A_c\circ \sigma =A_c$, so we deduce that the non-zero continuous solutions of \eqref{A2} are 
 $$
\left\{\begin{array}{l}
f : \left(\begin{matrix}
   a & b  \\
   0 & 1  \\
\end{matrix} 
 \right)\mapsto \alpha a^{\lambda}\\
g : \left(\begin{matrix}
   a & b  \\
   0 & 1  \\
\end{matrix} 
 \right)\mapsto a^{\lambda}
\end{array},\right.
$$
 where $\alpha \in \mathbb{C}\backslash \left\lbrace 0\right\rbrace $ and $\lambda \in \mathbb{C}$. 
\end{ex}
\begin{ex}
Let $S=(\mathbb{C},+)$  and  $\sigma (z)=2z$ for all $z\in \mathbb{C}$. The functional equation \eqref{A2} is written as follows 
\[f(z+z')+f(z+2z')=2f(z)g(z'),\quad z,z'\in S.\]
The continuous characters on $S$ are the functions of the form $\chi(z)=e^{a z}$, $z\in \mathbb{C}$,
where $a \in \mathbb{C}$. On the other hand according to \cite[Exersice 2.13]{ST} the continuous functions additive on $S$ are the functions : $A(z)=\lambda _1 z+\lambda_2 \overline{z}$,
  where $\lambda_1$ and $\lambda_2$ are complex constants, if $A\circ \sigma =-A$ then $A=0$, and $\chi =\chi\circ \sigma^2$ implies that $\chi=1$, so we deduce  that the continuous non-zero solutions of \eqref{A2} are
$$
\left\{\begin{array}{l}
f(z)=\alpha e^{a z}\\
g(z)=\dfrac{e^{a z}+e^{2a z}}{2}
\end{array},\right.
$$
where $\alpha \in \mathbb{C}\backslash \left\lbrace 0\right\rbrace $.
\end{ex}

\begin{ex}
Let $S=H_3$ be the Heisenberg group defined by
\[H_3=\left\lbrace \left( \begin{matrix}
   1 & x & z  \\
   0 & 1 & y  \\
   0 & 0 & 1  \\
\end{matrix} \right)\mid \quad x,y,z\in \mathbb{R} \right\rbrace . \]
We consider the following automorphism 
\[\sigma \left( \begin{matrix}
   1 & x & z  \\
   0 & 1 & y  \\
   0 & 0 & 1  \\
\end{matrix} \right)=\left( \begin{matrix}
   1 & x & 2z  \\
   0 & 1 & 2y  \\
   0 & 0 & 1  \\
\end{matrix} \right).\]
According to \cite[Example 2.11, Example 3.14]{ST}, the continuous  additive and the non-zero multiplicative functions on $S$ have respectively the forms
\[A \left( \begin{matrix}
   1 & x & z  \\
   0 & 1 & y  \\
   0 & 0 & 1  \\
\end{matrix} \right)=\alpha x+\beta y ,\]
and 
\[\chi \left( \begin{matrix}
   1 & x & z  \\
   0 & 1 & y  \\
   0 & 0 & 1  \\
\end{matrix} \right)=e^{ax+by},\]
where $\alpha ,\beta ,a ,b\in\mathbb{C}$. If $A\circ \sigma=-A$ then $A=0$. On the other hand $\chi \circ \sigma^2=\chi$ implies that $\chi =e^{ax}$. So the continuous non-zero solutions of equation \eqref{A2} are 
$$
\left\{\begin{array}{l}
f : \left(\begin{matrix}
   1 & x & z  \\
   0 & 1 & y  \\
   0 & 0 & 1  \\
\end{matrix} 
 \right)\mapsto \alpha e^{ax+by}\\
g : \left(\begin{matrix}
   1 & x & z  \\
   0 & 1 & y  \\
   0 & 0 & 1  \\
\end{matrix} 
 \right)\mapsto \dfrac{e^{ax+by}+e^{ax+2by}}{2} 
\end{array},\right.$$
where $\alpha\in \mathbb{C}\backslash \left\lbrace 0 \right\rbrace$. 
\end{ex}
\begin{ex}
Let $S=\left( \left]-1,1 \right[, . \right)$ and $\sigma (x)=x$ for all $x\in S$. $S$ is not generated by its squares and if $\chi$ is a continuous multiplicative function on $S$, then $\chi$ have one of the forms
\begin{equation}
\chi=1,\quad \chi (x)=
\left\{\begin{array}{l}
\lvert x\rvert^{\alpha},\quad x\neq 0 \\
0,\quad x= 0
\end{array}\right.
\quad \text{or}\quad \chi (x)=
\left\{\begin{array}{l}
\lvert x\rvert^{\alpha}sgn(x),\quad x\neq 0 \\
0,\quad x= 0,
\end{array}\right.
\label{ex1}
\end{equation} 
for some $\alpha \in \mathbb{C}$ such that $\mathcal{R}(\alpha )>0$, where $\mathcal{R}(\alpha )$ denote the real part of $\alpha$. The non-zero continuous solutions of \eqref{A2} are 
$$\left\{\begin{array}{l}
f(x)=c \chi (x) \\
g(x)=\chi (x)
\end{array},\right.$$
where $c\in \mathbb{C}\backslash \left\lbrace 0 \right\rbrace$ and $\chi$ have one of the forms in \eqref{ex1}.\\
\end{ex}
%\subsection*{Acknowledgment}
\subsection*{Declarations}

\textbf{Author contributions} This work is done  by the authors solely.\\
\\
\textbf{Funding} None.\\
\\
\textbf{Availability of data and materials} Not applicable.\\
\\
\textbf{Code Availability} Not Applicable.\\
\\
\textbf{Conflict of interest} None.

% ------------------------------------------------------------------------
\end{document}